\documentclass[a4paper, 10pt]{amsart}
\usepackage{amsmath, amscd}
\usepackage{amssymb, color, hyperref}

\def\doi#1{{\small\href{https://doi.org/#1}{\path{doi:#1}}}}
\def\arxiv#1{{\small\href{http://www.arxiv.org/abs/#1}{\path{arXiv:#1}}}}
\def\url#1{{\small\href{#1}{\path{#1}}}}

\theoremstyle{plain}
\newtheorem{theorem}{\bf Theorem}[section]
\newtheorem{proposition}[theorem]{\bf Proposition}
\newtheorem{lemma}[theorem]{\bf Lemma}

\theoremstyle{definition}

\newtheorem{remark}[theorem]{\bf Remark}

\newcommand{\N}{\mathbb N}
\newcommand{\Z}{\mathbb Z}
\newcommand{\R}{\mathbb R}

\newcommand{\C}{\mathbb C}

 \DeclareMathOperator{\ord}{ord}
 
 \DeclareMathOperator{\Ca}{\mathsf {Ca}}
\DeclareMathOperator{\spec}{spec}

\newcommand{\red}{{\text{\rm red}}}
\renewcommand{\t}{\, | \,}

\numberwithin{equation}{section}

\begin{document}

\title[On polynomial formulas for arithmetic invariants]{There is no polynomial formula for the catenary and the tame degree of finitely generated monoids}

\author{Alfred Geroldinger and Alessio Moscariello}

\address{University of Graz, NAWI Graz \\
Department of Mathematics and Scientific Computing \\
Heinrichstra{\ss}e 36\\
8010 Graz}

\email{alfred.geroldinger@uni-graz.at, alessio.moscariello@uni-graz.at}
\urladdr{https://imsc.uni-graz.at/geroldinger}


\keywords{finitely generated monoids, affine monoids, numerical monoids, catenary degree, tame degree}

\subjclass[2020]{20M13, 20M14, 13A05, 13F15}

\thanks{This work was supported  by the Austrian Science Fund FWF, Project PAT 9756623}

\begin{abstract}
In the last two decades there has been a wealth of results determining the precise value of the catenary degree and the tame degree. Mostly, however, only for  very special classes of monoids and domains. In the present work we now show that there is no polynomial formula, neither for the catenary nor for the tame degree, which is valid for a sufficiently large class of finitely generated monoids.
\end{abstract}

\maketitle

\section{Introduction} \label{1}

A (commutative and cancellative) monoid resp. a (commutative integral) domain is factorial if and only if it is a Krull monoid resp. a Krull domain with trivial class group. Factorization theory measures the deviation from the uniqueness of factorizations by arithmetic invariants, such as length sets, elasticities, catenary and tame degrees, and others. The overall goal is to relate arithmetic invariants and algebraic invariants of the objects under consideration to each other. 

\smallskip
Finitely generated monoids play a crucial role in Factorization Theory. They include (generalized) numerical monoids and affine monoids. Apart from being of interest in their own right, they also occur as target monoids of  larger classes of monoids and domains whose arithmetic invariants coincide with those of an associated finitely generated monoid (more on that in Subsection \ref{2.3}). For most invariants studied so far and which have a real number as their value, it is quite easy to show that they are finite for finitely generated monoids. However,  precise values (in terms of algebraic ingredients, such as the set of atoms) are known in very special cases only. 

\smallskip
The goal of the present note is to shed some light on this phenomenon and to show, as the title indicates, that there are no polynomial formulas for some of the key arithmetic invariants, which are valid for a sufficiently large class of finitely generated monoids. The model for our study is a paper by F.~Curtis \cite{Cu90a}, who showed that there is no polynomial formula for the Frobenius number of numerical monoids. We start by making precise what we mean by the existence of a polynomial formula. There are two  concepts of polynomial formulas and we will make use of both of them.

\smallskip
Let $H$ be a finitely generated monoid and (as outlined in Section \ref{2}) we may assume without restriction that $H$ is reduced. Then $H$ is contained in a finitely generated quotient group $\mathsf q (H)$, and we consider a monoid monomorphism $\varphi_H' \colon H \hookrightarrow G \times \Z^s$ where $s \in \N$ and $G$ is a finite abelian group with basis $({\mathsf e}_1,\ldots,{\mathsf e}_t)$ (note that we may choose $G$ to be trivial in the special case of affine monoids).  For an element $h \in H$, let $\varphi_H' (h) = (g, m_1, \ldots, m_s)$ and suppose that $g = a_1 \mathsf e_1 + \ldots, + a_t \mathsf e_t$ with $a_i \in [0, \ord (\mathsf e_i)-1]$ for $i \in [1,t]$.
Now let 
$\varphi_H: H \rightarrow \mathbb{Z}^t \times \mathbb{Z}^s$ be  defined by $\varphi_H(h)=(a_1,\ldots,a_t,m_1,\ldots,m_s)$.

\smallskip
Let $\mathcal H$ be a family of finitely generated reduced monoids and fix some positive integer $r \in \N$.  For each monoid $H$ in this family, let $\mathcal A (H)$ denote the set of atoms (which is unique) and suppose that  $|\mathcal A (H)| \le r$. Furthermore, let $\varphi_H: H \rightarrow \mathbb{Z}^t \times \mathbb{Z}^s$ be as above and let $\mathsf a(H) \in \mathbb{R}$ be an arithmetic invariant. We say that {\it there is a polynomial formula for $\mathsf a(H)$} (in terms of the atoms and valid for all monoids from $\mathcal H$)  if there exist polynomials $f_1,\ldots,f_\sigma \in \mathbb{C}[X_1,\ldots,X_{r(s+t)}]$ such that for every $H \in \mathcal{H}$  there is an $i \in [1,\sigma]$ such that $f_i(\varphi_H(h_1),\ldots,\varphi_H(h_r))=\mathsf a(H)$, where  $\mathcal A (H) = \{h_1,\ldots,h_{r'}\}$ and $\{h_{r'+1}, \ldots, h_r\} = \{1_H\}$ for $r' = |\mathcal A (H)|$.

\smallskip
We say that {\it there is an implicit polynomial formula}  for  $\mathsf a (H)$ 
(in terms of the atoms and valid for all monoids from $\mathcal H$) 
if there exists a nonzero polynomial $F \in \mathbb{C}[X_1,\ldots,X_{r(s+t)},Y]$ such that for every $H \in \mathcal{H}$ (and with $\{h_1,\ldots,h_r\}$ as above) we have $F (\varphi_H(h_1),\ldots,\varphi_H(h_r),\mathsf a(H))=0$. If the polynomials $f_1,\ldots,f_\sigma$ describe a polynomial formula for $\mathsf a(H)$, then 
\[
\displaystyle F = \prod_{i=1}^{\sigma} (f_i-Y)
\] 
is an implicit formula for $\mathsf a(H)$, and if there is no implicit formula for $\mathsf a(H)$, then there is no polynomial formula for $\mathsf a(H)$. However, the converse is not true (see Remark \ref{elasticity}). 
We formulate the main result of the present paper.

\smallskip
\begin{theorem} \label{main}
Let $r, s \in \N$, let $G$ be a finite abelian group,  and let $\mathcal{H}$ be a class of finitely generated reduced monoids which are embeddable into $(G \times \mathbb{Z}^s, +)$ and whose set of atoms has at most $r$ elements. If $\mathcal{H}$ contains all numerical monoids  generated by three atoms,
then there is no polynomial formula for the catenary degree  nor for the tame degree valid for all monoids in $\mathcal{H}$.
\end{theorem}

\smallskip
To analyze the assumptions of Theorem \ref{main}, let $H$ be a finitely generated reduced monoid. If $H$ has precisely one atom, then $H$ is isomorphic to $(\N_0, +)$, whence $H$ is factorial and $\mathsf c (H) = \mathsf t (H)=0$. If $H$ has precisely two atoms $u$ and $v$, then either $H$ is isomorphic to $(\N_0^2,+)$, whence $H$ is again factorial and $\mathsf c (H) = \mathsf t (H)=0$, or there exists a relation between the two atoms, say $u^{d_1}=v^{d_2}$ with integers $1 < d_1 \le  d_2$. In this case, we have $\mathsf c (H) = d_2$. Furthermore, if $\mathcal H$ would be a finite set of monoids, say $\mathcal H = \{H_1, \ldots, H_{\alpha}\}$, then there is a polynomial formula for the catenary degree (just set $f_i = X_1 - \mathsf c (H_i)$ for $i \in [1, \alpha]$).

In Section \ref{2}, we fix the required background necessary for the proof of Theorem \ref{main}. We recall the concepts of the catenary and the tame degree (Subsection \ref{2.1}), discuss finitely generated monoids (Subsection \ref{2.2}), and then we illustrate that Theorem \ref{main} applies to a larger class of monoids and domains which are not necessarily finitely generated (Subsection \ref{2.3}). The proof of Theorem \ref{main} will be given in Section \ref{3}. It is based on a result from projective geometry (Proposition \ref{criterion}) and on a result on generic presentations of semigroups (Proposition \ref{cat3}). In two final remarks we discuss phenomena which occur in settings close to the one in Theorem \ref{main}. Let $\mathcal H$ be a class of monoids as in Theorem \ref{main}. Although, there is no formula for the catenary and tame degree, there is an implicit formula for a different arithmetic invariant namely the elasticity (Remark \ref{elasticity}). Furthermore, although there is no formula for the catenary degree in terms of the atoms there are formulas in terms of other algebraic invariants (Remark \ref{tiles}).

\smallskip
\section{Preliminaries} \label{2}
\smallskip

We denote by $\N$ resp. by $\N_0$ the set of positive resp. of non-negative integers. For integers $a, b \in \Z$, we denote by $[a, b] = \{x \in \Z \mid a \le x \le b\}$ the discrete interval betwen $a$ and $b$. 
By a {\it monoid}, we mean a commutative cancellative semigroup with identity. Let $H$ be a monoid. Then $\mathsf q (H)$ denotes its quotient group, $\mathcal A (H)$ its set of atoms, $H^{\times}$  its group of units, and $H_{\red} = \{ a H^{\times} \mid a \in H \}$ its associated reduced monoid. For integers $h_1, \ldots, h_k \in \N_0$, we denote by $\langle h_1, \ldots, h_k \rangle \subseteq (\N_0, +)$ the submonoid generated by $h_1, \ldots, h_k$.
For a set $P$, let $\mathcal F (P)$ denote the (multiplicatively written) free abelian monoid with basis $P$.

\noindent
\subsection{Arithmetic Invariants.} \label{2.1}
Let $H$ be a monoid. Then $\mathsf Z (H) = \mathcal F ( \mathcal A (H_{\red}) )$ denotes the factorization monoid of $H$ and the homomorphism $\pi \colon \mathsf Z (H) \to H_{\red}$, defined by $\pi (u) = u$ for all $u \in \mathcal A (H_{\red})$, is the factorization homomorphism. For an element $a \in H$,  $\mathsf Z (a) = \pi^{-1} (aH^{\times})$ is the set of factorizations of $a$. Then $H$ is called 
\begin{itemize}
\item {\it atomic} if $\mathsf Z (a) \ne \emptyset$ for all $a \in H$ (equivalently, every $a \in H \setminus H^{\times}$ can be written as a finite nonempty product of atoms),
    
\item {\it factorial} if $|\mathsf Z (a)|=1$ for all $a \in H$ (equivalently, $H_{\red}$ is a free abelian monoid).
\end{itemize}    
To define the distance of factorizations, let $z, z' \in \mathsf Z (H)$ be given. Then we may write
\[
z = u_1 \cdot \ldots \cdot u_{\ell}v_1 \cdot \ldots \cdot v_m \quad \text{and} \quad z' = u_1 \cdot \ldots \cdot u_{\ell}w_1 \cdot \ldots \cdot w_n \,,
\]
where $\ell, m, n \in \N_0$ and all $u_i, v_j, w_k$ are atoms of $H_{\red}$ such that $\{v_1, \ldots, v_m\} \cap \{w_1, \ldots, w_n\} = \emptyset$. Then $\mathsf d (z, z') = \max \{m,n\} \in \N_0$ is called the {\it distance} between $z$ and $z'$.

Before defining the catenary degree, let us first consider an element $b \in H$ having two distinct factorizations $y$ and $y'$, say $y = v_1 \cdot \ldots \cdot v_m$ and $y' = w_1 \cdot \ldots \cdot w_n $, with $\mathsf d (z, z') = \max \{m ,n \}$. Then, for each $N \in \N$, the element $b^N$ has factorizations $y^N$ and ${y'}^N$ with $\mathsf d ( y^N, {y'}^N) = N \max \{m,n\}$. However, we can go from $y^N$ to ${y'}^N$ in small steps via the factorizations $y_i= y^i {y'}^{N-i}$ for $i \in [0, N]$. This phenomenon is captured by the concept of the catenary degree, whose definition we are recalling now. To do so, let $a \in H$ and $m \in \N_0$. A finite sequence $z_0, \ldots, z_k \in \mathsf Z (a)$ is called an $m$-chain of factorizations of a if $\mathsf d (z_{i-1}, z_i) \le m$ for all $i \in [1,k]$. We denote by $\mathsf c (a)$ the smallest $N \in \N_0 \cup \{\infty\}$ such that any two factorizations $z, z' \in \mathsf Z (a)$ can be concatenated by an $N$-chain. Finally,
\[
\Ca (H) =  \{\mathsf c (a) \mid a \in H, \mathsf c (a) > 0 \} \in \N_0 \cup \{\infty\} \quad \text{resp.} \quad \mathsf c (H) = \sup \Ca (H)
\]
denotes the set of positive catenary degrees resp. the  {\it catenary degree} of $H$. 

We refer to \cite[Chapter 1]{Ge-HK06a} and to \cite[Section 3]{Ga-Ge-Sc15a} for background information on the local and global tame degrees. To recall the definition, let $u \in H$ be an atom. Then the {\it local tame degree} $\mathsf t (H,u)$ is the smallest $N \in \N_0 \cup \{\infty\}$ with the following property:
\begin{enumerate}
\item[]
If \ $m \in \N$ and $v_1, \dots , v_m \in \mathcal A(H)$ are
such that \ $u \t v_1 \cdot \ldots \cdot v_m$, but $u$ divides no
proper subproduct of $v_1\cdot \ldots \cdot v_m$, then there exist
$\ell \in \N$ and  $u_2, \ldots , u_{\ell} \in \mathcal A
(H)$ such that $v_1 \cdot \ldots \cdot v_m = u u_2 \cdot \ldots
\cdot u_{\ell}$ and \ $\max \{\ell ,\, m \} \le N$.
\end{enumerate}
and
\[
\mathsf t (H) = \sup \{\mathsf t (H, u) \mid u \in \mathcal A (H) \} \in \N_0 \cup \{\infty\}
\]
denotes the {\it (global) tame degree} of $H$. We have $0 \le \mathsf c (H) \le \mathsf t (H)$, and  $H$ is factorial if and only if $\mathsf c (H)=0$ if and only if $\mathsf t (H)=0$. Catenary and tame degrees are among the best investigated invariants in factorization theory and there is an abundance of results. For a sample in the setting of numerical monoids see \cite{B-C-K-R06, C-G-L09, Om12a, Om-Ra09a, Ki-ON-Po16, C-C-M-M-P14, C-C-M-M-P17, ON-Pe18a, Ga-Ga-Ma23a, Su23a}; for computational methods see \cite{GG-MF-VT15}, and for results in other classes of monoids see Remark \ref{tiles}.

\smallskip
\noindent
\subsection{Finitely generated monoids and affine monoids.} \label{2.2}
As can be seen by the above definitions, invertible elements play no role in the definition of arithmetic invariants. In particular, we have $\mathsf c (H) = \mathsf c (H_{\red})$ and $\mathsf t (H) = \mathsf t ((H_{\red})$. Since we are interested in arithmetic invariants of monoids, we suppose from now on that monoids are reduced.

Let $H$ be a reduced monoid. For a subset $U \subseteq H$, the following conditions are equivalent (see \cite[Proposition 1.1.7]{Ge-HK06a})
\begin{enumerate}
\item[(a)] $H$ is atomic and $U = \mathcal A (H)$.

\item[(b)] $U$ is the smallest generating set of $H$.

\item[(c)] $U$ is a minimal generating set (with respect to set-theoretical inclusion).
\end{enumerate}
In particular, $H$ is finitely generated if and only if $H$ is atomic and $\mathcal A (H)$ is finite, and the set of atoms is the unique minimal generating set. 
An {\it affine monoid}  is a commutative, cancellative, finitely generated, and torsionfree semigroup with identity element. Thus, the quotient group of an affine monoid $H$ is a finitely generated free abelian group, whence there is a monoid monomorphism $\varphi \colon H \to (\Z^s, +)$ for some $s \in \N$. Affine monoids play a crucial role in combinatorial commutative algebra. They include numerical monoids and generalized numerical monoids, as introduced in \cite{Ci-Fa-Na25a}. 

\smallskip
\noindent
\subsection{More general monoids and domains} \label{2.3} There are monoids and domains which are not finitely generated but whose arithmetic invariants coincide with the ones of an associated finitely generated monoid. We sketch a couple of such situations. 

To begin with, let $H$ be an atomic monoid. If $P$ is any set of pairwise non-associated primes of $H$ and $T$ is the set of all elements of $H$ that are not divisible by any prime from $P$, then $T$ is a submonoid of $H$ and $H \cong \mathcal F (P) \times T$. This immediately implies that $\mathsf c (H) = \mathsf c (T)$ and $\mathsf t (H) = \mathsf t (T)$. Thus, Theorem \ref{main} applies to all atomic monoids which have only finitely many atoms that are not prime. Generalized Cohen-Kaplansky domains (introduced in \cite{An-An-Za92b}) are atomic domains with this property. Let $K$ be an algebraic number field with ring of integers $\mathcal O_K$ and let $\mathcal O \subseteq \mathcal O_K$ be an order. If for every nonzero prime ideal  $\mathfrak p \in \spec (\mathcal O)$ there is only one prime ideal $\mathfrak P \in \spec (\mathcal O_K)$ with $\mathfrak P \cap \mathcal O = \mathfrak p$, then the monoid $\mathcal I^* ( \mathcal O)$ of invertible prime ideals of $\mathcal O$ has only finitely many irreducible ideals that are not prime. This generalizes from orders in number fields to weakly Krull Mori domains (\cite[Theorem 3.7.10]{Ge-HK06a}).

Secondly, consider a Krull monoid $H$ with class group $G$ and let $G_P \subseteq G$ denote the set of classes containing prime divisors. If $G_P$ is finite, then the monoid of zero-sum sequences $\mathcal B (G_P)$ is a reduced finitely generated monoid. There is a transfer homomorphism $\boldsymbol \beta \colon H \to \mathcal B (G_P)$ and $\mathsf c (H) = \mathsf c ( \mathcal B (G_P))$ (apart from the exceptional case when  $\mathsf c ( \mathcal B (G_P)) =0$; see \cite[Theorem 3.4.10]{Ge-HK06a}).

\smallskip
\section{Proof of Theorem \ref{main}} \label{3}
\smallskip

This section is devoted to proving Theorem \ref{main}. We proceed in a series of lemmas and propositions. The first proposition (which is a modification of a part of \cite[Theorem 3]{Cu90a}) is a key ingredient. It provides sufficient conditions guaranteeing that there is no implicit formula for an invariant $\mathsf a(h_1,h_2,h_3) \in \mathbb{R}$, which depends on three integers $h_1,h_2,h_3$ and which is connected to some linear forms. 

We use some basic concepts from algebraic geometry. 

For $k \in \N$ and a homogeneous polynomial $F \in \C [X_0, \ldots, X_k]$, we denote by $\mathcal V (F) \subseteq \mathbb P_{\C}^k$ the vanishing locus of $F$, which is a projective variety. A consequence of the Nullstellensatz is that divisibility of polynomials and containment of vanishing loci are closely related: if $f,g \in \C[X_0,\ldots,X_k]$ are polynomials, with $g$ irreducible, we have that
\begin{equation}\label{div1}
	g | f \Longleftrightarrow \mathcal{V}(g) \subseteq \mathcal{V}(f).
\end{equation}
More in general, since $\C[X_0,\ldots,X_n]$ is factorial, if $g_1,\ldots,g_\sigma$ are irreducible distinct polynomials, then 
\begin{equation}\label{div+}
	g_1\ldots g_\sigma | f \Longleftrightarrow \mathcal{V}(g_i) \subseteq \mathcal{V}(f) \ \ \forall i=1,\ldots,\sigma .
\end{equation}   
             
These definition are given in the projective setting, but the equivalent definition hold in the affine space $\mathbb{A}^n(\mathbb{C})$ (dropping the assumption that $f$ is homogeneous). In particular, (\ref{div1}) and (\ref{div+}) hold in the affine setting as well.

We are ready to state the key ingredient of our argument.
\smallskip
\begin{proposition}\label{criterion}
Let $\mathcal P \subseteq \N_0$ be an infinite set, let $\beta \in \mathbb{R}_{>0}$ and let $\mathcal{H}$ be the class of all  submonoids of $\N_0$ generated by three elements. Let $\mathsf a(H) = \mathsf a (h_1,h_2,h_3) \in \R$ be an invariant, defined for every monoid $H = \langle h_1, h_2, h_3 \rangle \in \mathcal{H}$. 
Assume that for every $h_1 \in \mathcal P$ the following condition holds.
\begin{itemize}
\item For every $k \in \N_0$ with $k \le \beta h_1$, there exists an infinite set  $A_k \subseteq \mathbb{R}_{>0} \setminus \mathbb{Q}_{>0}$ of positive irrational numbers and a polynomial $\lambda_k \in \mathbb{C}[X_2,X_3]$ such that, for every $\alpha \in A_k$ and every $\epsilon > 0$, there are $h_2,h_3 \in \N_0$ such that
    \[
    \langle h_1, h_2, h_3 \rangle \in \mathcal H, \quad  \left| \alpha-\frac{h_3}{h_2} \right| < \epsilon,  \quad \text{ and} \quad  \mathsf a(h_1,h_2,h_3)=\lambda_k(h_2,h_3) \,.
    \] 
    Moreover, the polynomials $\lambda_k$ and $\lambda_{k'}$ are distinct if $k \neq k'$.
\end{itemize} 
Then there is no implicit polynomial formula $G(X_1,X_2,X_3,Y) \in \mathbb{C}[X_1,X_2,X_3,Y]$ for $\mathsf a(h_1,h_2,h_3)$.
\end{proposition}

\begin{proof}
Assume to the contrary that there exists an implicit polynomial formula $G(X_1,X_2,X_3,Y) \in \C[X_1, X_2, X_3, Y]$ for $\mathsf a (H)$. We fix some $h_1 \in \mathcal P$ and show that  $\deg G \ge \beta h_1$, a contradiction to the assumption that  $\mathcal P \subseteq \N_0$ is infinite.
	
We choose some $k \in \N_0$ with $k \le \beta h_1$, and set $F(X_2,X_3)=G(h_1,X_2,X_3,\lambda_k(X_2,X_3))$. 

Let $\alpha \in A_k$, $m \in \N$ and $\epsilon = \frac{1}{m}$. Then there exist $h_{2,m},h_{3,m} \in \N_0$ such that 
\[
\langle h_1, h_{2.m}, h_{3,m} \rangle \in \mathcal H , \quad 
\left|\alpha-\frac{h_{3,m}}{h_{2,m}} \right| < \frac{1}{m} \quad \text{ and } \quad \mathsf a(h_1,h_{2,m},h_{3,m})=\lambda_k(h_{2,m},h_{3,m}) \,.
\] 
Then, by construction, $F(h_{2,m},h_{3,m}) = 0$. Let $F^*(X_2,X_3,Z) \in \mathbb{C}[X_2,X_3,Z]$ be the homogenization of $F$ with respect to $Z$ and consider its  vanishing locus $\mathcal V(F^*) \subseteq \mathbb{P}^2(\mathbb{C})$. Then $F^*(1,\frac{h_{3,m}}{h_{2,m}},\frac{1}{h_{2,m}})=F^*(h_{2,m},h_{3,m},1)=0$. Since $\alpha \not \in \mathbb{Q}$, we have an infinite non-constant sequence of couples $(h_{2,m},h_{3,m})$, and thus considering its limit, we have by continuity $F^*(1,\alpha,0)=0$. 

Since $A_k$ is  infinite, it follows that $\mathcal V(F^*) \subseteq \mathbb{P}^2(\mathbb{C})$ must contain $\mathcal V(Z)$, whence $Z \t F^*(X_2,X_3,Z)$. Since $F^*$ is the homogenization of $F$, this implies that 
\[
F(X_2,X_3)=0 \quad \text{ and hence} \quad  G(h_1,X_2,X_3,\lambda_k(X_2,X_3))=0 \,.
\]
Now we consider $F(X_2,X_3,Y)=G(h_1,X_2,X_3,Y)$. Then $F$ vanishes on $\mathcal V(\lambda_k(X_2,X_3)-Y) \subseteq \mathbb{A}^3(\mathbb{C})$ for every $k \le \beta h_1$. But since the polynomials $\lambda_k(X_2,X_3)-Y \in \mathbb{C}[X_2,X_3,Y]$ are irreducible and pairwise linearly independent for all $k \le \beta h_1$, by (\ref{div+}) we deduce that the product of these polynomials divides $F$, and thus for the total degree we have $\deg G \ge \deg F \ge \beta h_1$ for every $h_1 \in \mathcal{P}$, a contradiction. 

\end{proof}

We continue with a special case of a result from \cite{BL-GA-GE11A}, which studies the interplay of generic presentations of a semigroup and its arithmetic invariants.

\smallskip
\begin{proposition} \label{cat3}
Let $H$ be a numerical monoid with $\mathcal{A}(H)=\{h_1,h_2,h_3\}$ where $h_1,h_2,h_3$ are pairwise coprime. For $\{i,j, \ell\}= [1,3]$, we define
\[
c_i=\min \{k \in \mathbb{N} \ | \ kh_i \in \langle h_j,h_{\ell} \rangle \} \quad \text{ and} \quad  c_ih_i = r_{i,j}h_j+r_{i,\ell}h_{\ell} \,, 
\]
where  $r_{i,j} \in \N_0$. Then   $c_i=r_{j,i}+r_{\ell,i}$  for all $i,j \in [1,3]$ and
$$\mathsf c(H)=\mathsf t(H)=\max\{c_1,c_2,c_3,r_{12}+r_{13},r_{21}+r_{23},r_{31}+r_{32}\}.$$
\end{proposition}

\begin{proof}
Since $h_1, h_2, h_3$ are pairwise coprime, it follows that all $r_{i,j} > 0$. Thus, the assertion is a special case of \cite[Corollary 5.8]{BL-GA-GE11A}.
\end{proof}

Our next goal is to  construct an infinite family of numerical monoids generated by three atoms which satisfy the conditions of Proposition \ref{criterion}. 
To do so we need a technical lemma which is a simple consequence of Dirichlet's Prime Number Theorem (see \cite[Lemma 1]{Cu90a}).

\smallskip
\begin{lemma} \label{hyp}
	Let $\alpha$ and $\epsilon$ be positive real numbers, let $p  \in \N$ be an odd prime, and let $i,j \in \N_0$ be  coprime with $p$. Then there exist $x,y \in \N_0$ such that $x$ is prime, $x \equiv i \mod{p}$, $y \equiv j \mod{p}$, $\gcd(x,y)=1$ and $|\alpha-\frac{y}{x}| < \epsilon$.
\end{lemma}

The three integers $p,x,y$ described in Lemma \ref{hyp} are pairwise coprime, and the numerical monoid $H= \langle p,x,y \rangle$ satisfies the hypotheses of Proposition \ref{cat3}. Thus, we can compute $\mathsf c(H)$ and $\mathsf t(H)$.

\smallskip
\begin{proposition}\label{comp3}
Let $H$ be the numerical monoid with $\mathcal A (H) = \{h_1, h_2, h_3\}$, where 
$h_1,h_2$ are odd prime numbers, $h_3$ is coprime with $h_1$ and $h_2$, and $2 < h_1 < h_2 < h_3$. Let $k \in \N$ with $2 \le k \le \frac{h_1-1}{2}$ and suppose that 
\[
h_1-k < \frac{h_3}{h_2} < h_1-k+1, \quad   h_2 \equiv 1 \mod h_1, \quad \text{and} \quad h_3 \equiv h_1-k+1 \mod p \,.
\]
Then, for $i \in [1,3]$,  there are non-constant distinct linear forms $\lambda_{k,i}: \mathbb{C}^2 \rightarrow \mathbb{C}$ (with coefficients depending on $h_1$ and $k$), such that  $$\mathsf c(H) = \mathsf t(H) = \max\{\lambda_{k,i}(h_2,h_3) \ | \ i \in [1,3] \}.$$
\end{proposition}

\begin{proof}
We compute the values of $c_i$ and $r_{ij},r_{il}$.
Since $h_1-k+1 \ge \frac{h_1-1}{2}+1$, a quick computation modulo $h_1$ returns $c_3=2$ and $r_{32}=h_1-2k+2$; since $r_{i,j} > 0$ for every $i,j$ and $c_3=r_{13}+r_{23}$, we have $r_{13}=r_{23}=1$. Then $c_2h_2 = h_3 + r_{21}h_1$, and again working modulo $h_1$ we obtain $c_2=h_1-k+1$. Since $c_2=r_{12}+r_{32}$ we further deduce $r_{12}=k-1$. Finally, by substituting these results in the equations $c_ih_i=r_{ij}h_j+r_{il}h_l$ we can deduce $c_1=\frac{(k-1)h_2+h_3}{h_1}$, $r_{21}=h_2-\frac{(k-1)h_2+h_3}{h_1}$ and $r_{31}=\frac{2h_3-(h_1-2k+2)h_2}{h_1}$. Therefore in the formula of Proposition \ref{cat3} we obtain 
\begin{eqnarray*}
	\lambda_1(h_2,h_3) & := &  c_1 =  \frac{(k-1)h_2+h_3}{h_1} \\
	\lambda_2(h_2,h_3) & := & c_2  =  h_1-k+1 \\
	\lambda_3(h_2,h_3) & := & c_3 = 2 \\
	\lambda_4(h_2,h_3) & := & r_{12}+r_{13}  =  k \\
	\lambda_5(h_2,h_3) & := & r_{21}+r_{23} = h_2-\frac{(k-1)h_2+h_3}{h_1}+1 \\
	\lambda_6(h_2,h_3) & := & r_{31}+r_{32} = \frac{2h_3-(h_1-2k+2)h_2}{h_1}+(h_1-2k+2) 
\end{eqnarray*}
Since $k \le \frac{h_1-1}{2}$ we have $2 \le k < h_1-k+1$. Moreover, since $h_1,h_2$ are odd primes and $h_2 \equiv 1 \mod h_1$ we have $h_2 \ge 2h_1+1$. Then, since $\lambda_5(h_2,h_3)+\lambda_1(h_2,h_3)=h_2$, we have $\max \{\lambda_5(h_2,h_3),\lambda_1(h_2,h_3)\} \ge \frac{\lambda_5(h_2,h_3)+\lambda_1(h_2,h_3)}{2} = \frac{h_2}{2} > h_1 > h_1-k+1$.
Thus, from the formula of Proposition \ref{cat3} we obtain 
$$\mathsf c(H) = \mathsf t(H) = \max\{\lambda_1(h_2,h_3),\lambda_5(h_2,h_3),\lambda_6(h_2,h_3)\},$$
proving the thesis.
\end{proof}

We are ready to prove our main result for the class of numerical monoids having precisely three atoms.

\begin{proposition} \label{noformulanum}
There is no  polynomial formula neither for the catenary degree nor for the tame degree, valid  for all numerical monoids having precisely three atoms.
\end{proposition}

\begin{proof}
Let $\mathcal P$ denote the set of odd prime numbers. For any numerical monoid $H$, let 
 $\mathsf a (H)$ be either the catenary or the tame degree.

Let $h_1 \in \mathcal P$, $k \in \N$ with  $2 \le k \le \frac{h_1-1}{2}$, and let  $A_k= \{ x \in \R \mid h_1-k < x < h_1-k+1\}$. By Lemma \ref{hyp}, for every $\alpha \in A_k$ and every $\epsilon > 0$ there exist $h_2,h_3 \in \N_0$ such that $h_1,h_2,h_3$ are pairwise coprime, 
\[
|\alpha-\frac{h_3}{h_2}| < \epsilon,  \quad h_2 \equiv 1 \mod{h_1}, \quad \text{ and } \quad h_3 \equiv h_1-k+1 \mod{h_1} \,.
\] 
Then, by Proposition \ref{comp3}, there are three distinct, non-constant linear forms $\lambda_{k,1}, \lambda_{k,2}, \lambda_{k,3}$ (with coefficients in $\mathbb{C}$ depending on $h_1$ and $k$) such that $\mathsf a(H) = \max\{\lambda_{k,i}(h_2,h_3) \ | \ i \in [1,3] \}$. 

Thus there are three distinct subsets $U_1,U_2,U_3$ of the set 
\[
U = \{(h_1,h_2,h_3) \ | \ 2 < h_1 < h_2 < h_3 \text{ are pairwise coprime positive integers}\}
\]
such that in $U_i$ we have $\mathsf a(H)=\lambda_{k,i}(h_2,h_3)$. Since all the $\lambda_{k,i}(h_2,h_3)$ satisfy the hypothesis of Proposition \ref{criterion} and coincide with $\mathsf a(H)$ on $U_i$, it follows by Proposition \ref{criterion} that there is no implicit polynomial formula and hence no polynomial formula for $\mathsf a(H)$.
\end{proof}

\smallskip
\begin{proof}[Proof of Theorem \ref{main}]
Let $\mathcal H$ be a class of finitely generated reduced monoids as given in Theorem \ref{main} and suppose that $\mathcal{H}$ contains all numerical monoids  generated by three atoms,
By Proposition \ref{noformulanum}, there is no polynomial formula, neither for the catenary nor for the tame degree, which is valid for all numerical monoids generated by three atoms, whence there is no such formula for the larger class of monoids $\mathcal H$.
\end{proof}

We end with two remarks highlighting that the statement of Theorem \ref{main} does not hold true for other arithmetic invariants and that there are formulas, say for the catenary degree, in algebraic terms other than the atoms.

\smallskip
\begin{remark} \label{elasticity}
We showed that there is no polynomial formula for the catenary degree or the tame degree (valid for all numerical monoids with three generators) by showing that there is no implicit polynomial formula. In contrast to this, there is an implicit polynomial formula for a different invariant, namely for the elasticity, which is valid for a class of monoids which includes all numerical monoids. 

For an atomic monoid $H$,
\[
\rho (H) = \sup \{ \ell /k \mid \ \text{there are atoms $u_1, \ldots, u_k, v_1, \ldots, v_{\ell}$ with $u_1 \cdot \ldots \cdot u_k = v_1 \cdot \ldots \cdot u_{\ell}$} \} \in \R_{\ge 1} \cup \{\infty\}
\]
denotes  the {\it elasticity} of $H$, which is one of the best investigated invariants in factorization theory (see \cite{Ge-HK06a} for some background). A monoid $H$ is {\it finitely primary} (of rank $s \in \N$ and of exponent $\alpha \in \N$) if it is a submonoid of a factorial monoid $F = F^{\times} \times \mathcal F (\{p_1, \ldots, p_s\})$ such that
\[
H \setminus H^{\times} \subseteq p_1 \cdot \ldots \cdot p_sF \quad \text{and} \quad (p_1 \cdot \ldots \cdot p_s)^{\alpha}F \subseteq H \,.
\]
Then $\rho (H) < \infty$ if and only if $s=1$, and $H$ is finitely generated if and only if $s=1$ and $F^{\times}/H^{\times}$ is finite. If $D$ is a one-dimensional local Noetherian domain whose integral closure is a finitely generated $D$-module, then the multiplicative monoid $D \setminus \{0\}$ is finitely primary. Note that all numerical monoids are finitely primary of rank one. 

Suppose that $H$ is a reduced, finitely primary, and  finitely generated monoid, and (with the notation of above)  let $\mathcal A (H) = \{\varepsilon_1 p_1^{n_1}, \ldots, \varepsilon_s p_1^{n_s} \}$, where $\varepsilon_1, \ldots, \varepsilon_s \in F^{\times}$ and $1 \le n_1 < \ldots < n_s$. Then $\rho (H) = n_s/n_1$ by 
 \cite[Lemma 4.1]{Ge-Zh18a}.  Therefore, there is no polynomial formula for the elasticity (valid for all reduced, finitely primary, and finitely generated monoids), but there is an implicit polynomial formula for $\rho(H)$, given by the polynomial $G=(X_s-X_1Y)$.  
\end{remark}

\smallskip
\begin{remark} \label{tiles}
By Theorem \ref{main},  there is no polynomial formula for the catenary degree in terms of the atoms, valid for a given class of monoids $\mathcal H$. However, there are formulas for the catenary degree valid for suitable large classes of monoids, which are given by algebraic parameters other than the atoms. We provide some examples.

1. There is a formula for the catenary degree, valid for all numerical monoids with three atoms, in graph theoretical language (\cite[Proposition 5]{AG-GS10}).

2. If a monoid is given in terms of generators and relations, then there is a formula in terms of relations (see \cite[Corollary 9]{Ph10a}, \cite{Ph15a}). Proposition \ref{cat3} goes back to a formula of this type.

3. Let $G$ be a  finite abelian group and, to exclude trivial cases, suppose that $|G| \ge 3$. Then the monoid $\mathcal B (G)$ of zero-sum sequences over $G$ is an affine monoid (indeed, it is a normal affine monoid, in other words, a Krull monoid). If $H$ is a Krull monoid with class group $G$ and every class contains a prime divisor, then the catenary degrees of $H$ and of $\mathcal B (G)$ coincide. There are formulas for the catenary degree in terms of the group invariants, valid for several classes of finite abelian groups (e.g., \cite{Ge-Zh15b};  for a survey and  background on catenary degrees in Krull monoids \cite{Ge-HK06a, Sc16a, Gr22a}).

4. For a sample of results on catenary degrees in various ring theoretic settings, see \cite{Br-Ge-Re20, Do20a,MR4613278}.
\end{remark} 


\providecommand{\bysame}{\leavevmode\hbox to3em{\hrulefill}\thinspace}
\providecommand{\MR}{\relax\ifhmode\unskip\space\fi MR }
\providecommand{\MRhref}[2]{%
  \href{http://www.ams.org/mathscinet-getitem?mr=#1}{#2}
}
\providecommand{\href}[2]{#2}

\end{document}